\documentclass[11pt]  {amsart} 
\usepackage{amsmath,latexsym,amssymb,amsmath,
amscd,amsthm,amsxtra}

\newtheorem{theorem}{Theorem}[section]
\newtheorem{proposition}{Proposition}[section]
\newtheorem{lemma}{Lemma}[section]

\def\l{\lambda}

\def\<{\langle}
\def\>{\rangle}

\def \ds{\displaystyle}
\def \vs{\vspace*{0.1cm}}

\begin{document}

\title[An optimal anisotropic Poincar\'e inequality ]
{An optimal anisotropic Poincar\'e inequality for
convex domains}
\author{Guofang Wang}
\address{ Albert-Ludwigs-Universit\"at Freiburg,
Mathematisches Institut,
Eckerstr. 1,
79104 Freiburg, Germany}
\email{guofang.wang@math.uni-freiburg.de}
\thanks{GW is partly supported by SFB/TR71 ``Geometric partial differential equations'' of DFG. CX is supported by China Scholarship Council.}
\author{Chao Xia}\address{Albert-Ludwigs-Universit\"at Freiburg, Mathematisches Institut,  Eckerstr. 1, 79104 Freiburg, Germany}
\email{chao.xia@math.uni-freiburg.de}

\date{}

\begin{abstract}{ In this paper, we prove a sharp  lower bound  
 of the first (nonzero) eigenvalue of Finsler-Laplacian  with the  Neumann boundary condition. Equivalently, we prove an optimal anisotropic Poincar\'e inequality
for convex domains, which generalizes the result of Payne-Weinberger \cite{PW}. A  lower bound
 of the first (nonzero) eigenvalue of Finsler-Laplacian  with the  Dirichlet boundary condition is also proved.

\

\noindent
{\rm Keywords.} Finsler-Laplacian, first eigenvalue, gradient estimate,
optimal Poincar\'e inequality.
}

\end{abstract}

\medskip

\maketitle

\section{Introduction and main results}

In this paper we are interested in studying the eigenvalues of the Finsler-Laplacian $Q$, which is a
natural generalization of the ordinary Laplacian $\Delta$.
We say that $F$ is a {\it norm} on $\mathbb{R}^n$, if $F: \mathbb{R}^n\rightarrow [0,+\infty)$ is a convex function of
class $C^1(\mathbb{R}^n\backslash\{0\})$, which is even and
positively $1$-homogeneous, i.e.
\begin{equation*}
\label{01}
F(t\xi)=|t|F(\xi)
\hbox{ for any }t\in\mathbb{R},\quad\xi\in \mathbb{R}^n,\end{equation*} and
\begin{equation*} \label{02}F(\xi)>0 \hbox{  for any  }\xi\neq 0.\end{equation*}
A typical norm on $\mathbb{R}^n$ is
$F(\xi)= (\sum_{i=1}^n | \xi_i|^q)^{1/q}$ for $q\in (1,\infty)$. 
The {\it Finsler-Laplacian} on $(\mathbb{R}^n,F)$ is defined by
\begin{eqnarray}\label{Finsler Laplacian}
Qu:=\sum_{i=1}^n \frac {\partial}{\partial x_i}\left( F(\nabla u)
 F_{\xi_i}(\nabla u)\right)=\sum_{i=1}^n \frac {\partial}{\partial x_i}\left( \frac {\partial}{\partial \xi_i}\left(\frac12 F^2\right)(\nabla u)\right),
\end{eqnarray}
 where $ F_{\xi_i}=\frac{\partial F}{\partial \xi_i}
.$ When $F(\xi)=|\xi| =(\sum_{i=1}^n | \xi_i|^2)^{1/2}$, the Finsler-Laplacian $Q=\Delta$, the usual Laplacian.



The Finsler-Laplacian has been studied by many mathematicians,
 both in the context of Finsler geometry (see e.g. \cite{AB, GS, Oh, OS, S}) and quasilinear PDE ( see e.g. \cite{AFTL, BFK, FK, WX1, WX2, WX3}). Especially, many problems related to  the first eigenvalue of Finsler-Laplacian have been already considered in \cite{BFK, GS, K, Oh, WX2}.
In this paper we investigate the estimates of the first eigenvalue of the Finsler-Laplacian.

Let $\Omega$ be a smooth bounded domain in $\mathbb{R}^n$ and $\nu$ be the outward normal of its boundary $\partial\Omega$. The first eigenvalue $\lambda_1$ of Finsler-Laplacian $Q$ is defined by
the smallest positive constant such that there exists a nonconstant function $u$ satisfying
\begin{eqnarray}\label{first eigenvalue}-Qu=\lambda_1 u \quad \hbox{ in } \Omega\end{eqnarray}
with the Dirichlet boundary condition
\begin{eqnarray}\label{D boundary}u=0\hbox{ on } \partial \Omega\end{eqnarray}
or the Neumann boundary condition
\begin{eqnarray}\label{N boundary}\left\langle F_{\xi} (\nabla u),\nu\right\rangle=0 \hbox{ on } \partial \Omega.\end{eqnarray}
We call  $\l_1$ the {\it first Dirichlet  eigenvalue } (the {\it first Neumann eigenvalue } resp.)
and denote it by $\lambda_1^D$ (by $\lambda_1^N$ resp.).
Here $\left\langle F_{\xi} (\nabla u),\nu\right\rangle=\sum_{i=1}^n F_{\xi_i}(\nabla u) \nu^i$ and $\nu=(\nu^1,\cdots,\nu^n)$.
(\ref{N boundary}) is a natural
 Neumann boundary condition for the Finsler Laplacian.
 When $F(\xi)=|\xi|$, $\left\langle F_{\xi} (\nabla u),\nu\right\rangle=\frac {\partial u}{\partial \nu}.$

The first Dirichlet (Neumann, resp.) eigenvalue can be formulated as  a variational problem by
\begin{eqnarray}\lambda_1^D(\Omega)=\inf \left\{\frac{\int_\Omega F^2(\nabla u)dx}{\int_\Omega u^2 dx}\, \bigg|\,  0\neq u\in W_0^{1,2}(\Omega)\right\}.\end{eqnarray}

\begin{eqnarray}\lambda_1^N(\Omega)=\inf \left\{\frac{\int_\Omega F^2(\nabla u)dx}{\int_\Omega u^2 dx}\, \bigg| \, 0\neq u\in W^{1,2}(\Omega), \int_\Omega u dx=0\right\}.\end{eqnarray}
Therefore obtaining
a sharp estimate of first eigenvalue is equivalent to obtaining the best constant in Poincar\'e type inequalities.

We remark that equation (\ref{first eigenvalue}) should be understood in a weak sense, i.e.
\begin{eqnarray*}\int_\Omega \sum_{i=1}^n \frac {\partial}{\partial \xi_i}\left(\frac12 F^2\right)(\nabla u)\varphi_i dx=\int_\Omega \lambda_1 u\varphi dx\hbox{ for any $\varphi\in C^\infty_0(\Omega)$}.\end{eqnarray*}

\
Finding a lower bound for the first eigenvalue is always an interesting problem. In \cite{BFK, GS}, the authors proved the Faber-Krahn type inequality for the first Dirichlet eigenvalue of the Finsler-Laplacian. A Cheeger type estimate for the first eigenvalue of the Finsler-Laplacian involving isoperimetric constant was also obtained there. In this paper, we are interested in the Payne-Weinberger type sharp estimate \cite{PW} of the first eigenvalue
in terms of some geometric quantity, such as the diameter with respect to $F$.

Before stating our main result, we need to introduce some concepts and definitions.
We say that $\partial\Omega$ is  {\it weakly convex} if the second fundamental form of  $\partial\Omega$ is nonnegative definite. We say that $\partial\Omega$ is  {\it $F$-mean convex} if the  $F$-mean curvature $H_F$ is nonnegative. For the definition of $F$-mean curvature, see section \ref{Pre}.

There is another convex function $F^0$ related to $F$, which  is defined to be
the support function of $K:=\{x\in\mathbb{R}^n :
F(x)<1\}$, namely
\[
{F^0}(x):=\sup_{\xi\in K}\langle x,\xi\rangle.
\]
It is easy to verify that ${F^0}:
\mathbb{R}^n\mapsto[0,+\infty)$ is also  a convex, even, $1$-positively homogeneous
function of class $C^1(\mathbb{R}^n\backslash\{0\})$. Actually
${F^0}$ is dual to $F$ (see for instance \cite{AFTL}) in the
sense that
$${F^0}(x)= \sup_{\xi\neq 0} \frac{\langle x,\xi\rangle}{F(\xi)}\quad \hbox{and}\quad F(x)= \sup_{\xi\neq 0}
\frac{\langle x,\xi\rangle}{{F^0}(\xi)}.$$
Hence the Cauchy-Schwarz inequality holds in the sense that
\begin{eqnarray}\label{CS}\langle\xi,\eta\rangle_{\mathbb{R}^n}\leq F(\xi)F^0(\eta).\end{eqnarray}
We call $\mathcal{W}_r(x_0):=\{x\in {\mathbb R}^n\,|\,F^0(x-x_0)\leq r\}$ 
a {\it  Wulff ball} of radius $r$
with center at $x_0$ .
We say $\gamma:[0,1]\to \Omega$ a {\it  minimal geodesic} from $x_1$ to $x_2$ if  $$d_F(x_1,x_2):=\int_0^1 F^0(\dot{\gamma}(t))dt=\inf \int_0^1 F^0(\dot{\tilde{\gamma}}(t))dt,$$
where the infimum takes on all $C^1$ curves $\tilde{\gamma}(t)$ in $\Omega$ from $x_1$ to $x_2$. In fact $\gamma$ is a straight line and $d_F(x_1,x_2)=F^0(x_2-x_1).$ We call $d_F(x_1,x_2)$ the {\it $F$-distance} between $x_1$ and $x_2$. 

Now we can  define the {\it diameter} $d_F$ of $\Omega$ with respect to the norm $F$ on $\mathbb{R}^n$ as
\begin{equation*}
d_F:=\sup_{x_1,x_2\in\overline{\Omega}} d_F(x_1,x_2).\end{equation*}
In the same spirit we define the {\it inscribed radius} $i_F$ of $\Omega$ with respect to the norm $F$ on $\mathbb{R}^n$ as the radius of the biggest Wulff ball that can be enclosed in $\overline{\Omega}$.

\

Our main result of this paper is
\begin{theorem}\label{main thm2}Let $\Omega$ be a smooth bounded domain in $\mathbb{R}^n$ and $F\in C^1(\mathbb{R}^n\setminus\{0\})$  be a norm on $\mathbb{R}^n$. 
Let $\lambda_1^N$ be the first Neumann eigenvalue of the Finsler-Laplacian (\ref{Finsler Laplacian}). 
Assume  that $\partial\Omega$ is weakly convex. Then $\lambda_1^N$ satisfies
\begin{equation}\label{N}
\lambda_1^N\geq\frac{\pi^2}{d_F^2}.\end{equation}
Moreover, equality in (\ref{N}) holds if and only if $\Omega$ is a segment in $\mathbb{R}$.
\end{theorem}

Estimate (\ref{N}) for the Neumann boundary problem is optimal. This is in fact a
generalization of the classical result of Payne-Weinberger in \cite{PW} on an optimal estimate of the first Neumann eigenvalue
of the ordinary Laplacian.
See also \cite{Be}. There are many interesting generalizations. Here we just mention its generalization to Riemannian manifolds, since we will use the methods developed there.
It should be also interesting to ask if the methods of \cite{PW} and \cite{Be}
 work to reprove our result, since there are lots of motivations in computational mathematics.

For a smooth compact $n$-dimensional Riemannian manifold $(M,g)$ with nonnegative Ricci curvature and diameter $d$, possibly with boundary, the first Neumann eigenvalue $\lambda_1$ of Laplace operator $\Delta$ is defined to be the smallest positive constant such that there is a nonconstant function $u$ satisfying
\[-\Delta u=\lambda_1 u \hbox{ in } M,\]
with
 \[\frac{\partial u}{\partial \nu} =0 \hbox{ on } \partial M,\]
 if $\partial M$ is not empty,
where $\nu$ denotes the outward normal of $\partial M$.
A  fundamental work of  Li \cite{Li}, Li-Yau \cite{LY}, Zhong-Yang \cite{ZY}
gives us the following optimal estimate
\begin{equation}
\label{best}\lambda_1\geq \frac{\pi^2}{d^2},
\end{equation}
where $d$ is the diameter of $M$ with respect to $g$.
Li-Yau \cite{LY} derived a gradient estimate for the eigenfunction $u$ and proved that $\lambda_1\geq \frac{\pi^2}{4d^2} $ and
Li \cite{Li} used another auxiliary function to obtain a better estimate $\lambda_1\geq \frac{\pi^2}{2d^2}.$ Finally, Zhong-Yang \cite{ZY} was able to use a more precise auxiliary function to get the sharp estimate $\lambda_1\geq \frac{\pi^2}{d^2}$, which is optimal in the sense that the lower bound is achieved  by a  circle or a segment.  Recently Hang-Wang  \cite{HW} proved that equality in (\ref{best}) holds if and only if $M$ is a circle or a segment. For the related work see also \cite{Kr}, \cite{CW} and \cite{BQ}. Very recently these results
were generalized to the $p$-Laplacian in \cite{Va} and to the Laplacian
on Alexandrov spaces in \cite{QZZ}.

\

For the  Dirichlet problem we have
\begin{theorem}\label{main thm} Let $\Omega$ be a smooth bounded domain in $\mathbb{R}^n$ and $F\in C^1(\mathbb{R}^n\setminus\{0\})$  be a norm on $\mathbb{R}^n$. 
Assume that  $\lambda_1^D$ are the first Dirichlet eigenvalue of the Finsler-Laplacian (\ref{Finsler Laplacian}). 
Assume further that $\partial\Omega$ is F-mean convex. Then $\lambda_1^N$ satisfies
\begin{equation}\label{D}
\lambda_1^D\geq\frac{\pi^2}{4i_F^2}.\end{equation}
\end{theorem}

Estimate (\ref{D}) is by no mean optimal.

Our idea to prove the result on the Dirichlet eigenvalue is based on the gradient estimate technique for eigenfunctions of Li-Yau \cite{Li, LY}. This idea also works for the first Neumann eigenvalue to get a rough estimate, say $\lambda_1^N\geq\frac{\pi^2}{2d_F^2}$. However, for getting the sharp estimate of the first Neumann eigenvalue (\ref{N}), the method of Zhong-Yang seems hard to apply. Instead, we adopt the technique based on gradient comparison with a one dimensional model function, which was developed by Kr\"oger \cite{Kr} and improved by Chen-Wang \cite{CW} and Bakry-Qian \cite{BQ}.
Surprisingly, we find that the one dimensional model coincides with that for the Laplacian case. In fact, this must be the case because when we consider $F$ in $\mathbb{R}$, it can only be $F(x)=c|x|$ with $c>0$, a multiple of the standard Euclidean norm. In order to  get the gradient comparison theorem,  we need a
Bochner type formula (\ref{Bochner}), A Kato type inequality (\ref{Kato}) and
a refined inequality (\ref{eq1}), which was referred to as the ``extended Curvature-Dimension inequality" in the context of Bakry-Qian \cite{BQ}. Interestingly,  the proof of 
these inequalities sounds  more ``naturally'' than the proof of their counterpart for
the usual Laplace operator.
 These inequalities may have their own interest.  Another difficulty we encounter
 is to handle the boundary maximum due to  the different representation of the Neumann boundary condition (\ref{N boundary}). We find a suitable vector field $V$ (see its explicit construction in Section \ref{NFG}) to avoid this difficulty. With the gradient comparison theorem, we are able to follow step by step the work of Bakry-Qian \cite{BQ} to get the sharp estimate. The proof for the rigidity part of Theorem \ref{main thm2} follows closely the work of Hang-Wang \cite{HW}. Here we need pay more attention on the points with vanishing $|\nabla u|$.

A natural question arises whether one can generalize Theorem \ref{main thm2} to manifolds?
The Finsler-Laplacian with the norm $F$ has not a direct generalization to Riemannian manifolds. However, it has a (natural) generalization to Finsler manifolds.
In fact, ${\mathbb R}^n$ with $F$ can be viewed as  a special Finsler manifold.
On a general Finsler manifold, there is a generalized Finsler-Laplacian, see for instance
\cite{GS, Oh, S}.
A Lichnerowicz type result for the first eigenvalue of this Laplacian  was obtained
in \cite{Oh} under a condition on
 some kind of  new Ricci curvature $Ric_N,N\in[n,\infty]$. 
 A Li-Yau-Zhong-Yang type sharp estimate, i.e., a generalization of Theorem \ref{main thm2} for this generalized Laplacian on Finsler manifolds would be a challenging problem.
 We  will study this problem
in a forthcoming paper.

\

The paper is organized as follows. In Section \ref{Pre}, we give some preliminary results on $1$-homogeneous convex functions and
the $F$-mean curvature
  and prove useful inequalities.  In Section 3 we prove the sharp estimate for the first Neumann  eigenvalue and classify the equality case.
We handle the first Dirichlet eigenvalue in Section \ref{DFG}.

\

\section{Preliminary}\label{Pre}

Without of loss generality, we may assume that $F\in C^3(\mathbb{R}^n\setminus\{0\})$ and 
$F$ is a {\it strongly convex norm} on $\mathbb{R}^n$, i.e. $F$ satisfies 
\begin{equation*} \hbox{Hess}(F^2) \hbox{ is positive definite in } \mathbb{R}^n\setminus\{0\}.
\end{equation*}
In fact, 
for any norm $F\in C^1(\mathbb{R}^n\setminus\{0\})$,  there exists a sequence $F_\varepsilon\in C^3(\mathbb{R}^n\setminus\{0\})$ such that the strongly convex  norm $\widetilde{F}_\varepsilon:=\sqrt{F_\varepsilon^2+\varepsilon|x|^2}$ converges  to $F$ uniformly in $C_{loc}^1(\mathbb{R}^n\setminus\{0\})$, then the corresponding first eigenvalue  $(\lambda_1)_\varepsilon$ of Finsler-Laplacian with respect to $\widetilde{F}_\varepsilon$, converges to $\lambda_1$ as well. Here $|\cdot|$ denotes the Euclidean norm.

Therefore, in the following sections, we assume that $F\in C^3(\mathbb{R}^n\setminus\{0\})$ and
$F$ is a  strongly convex norm on $\mathbb{R}^n$. Thus (\ref{first eigenvalue}) is degenerate elliptic among $\Omega$ and uniformly elliptic in $\Omega\setminus \mathcal{C},$
 where $\mathcal{C}:=\{x\in\Omega|\nabla u(x)=0\}$ denotes the set of degenerate points.
The standard regularity theory for degenerate elliptic equation (see e.g. \cite{BFK, To}) implies that
$u\in C^{1,\alpha}(\Omega)\bigcap  C^{2,\alpha}(\Omega\setminus \mathcal{C})$.
 
The following property is an obvious consequence of 1-homogeneity of $F$.
\begin{proposition}\label{property}
Let $F: \mathbb{R}^n\rightarrow[0,+\infty)$ be a 1-homogeneous
function, then the following holds:
\begin{itemize}
\item[(i)] $\sum_{i=1}^n F_{\xi_i}(\xi)\xi_i=F(\xi)$;
\item[(ii)] $\sum_{j=1}^n F_{\xi_i\xi_j}(\xi)\xi_j=0$, for any
$i=1,2,\ldots,n$;
\end{itemize}
\end{proposition}\qed

For simplicity, from now on  we will follow the summation convention and frequently use the notations  $F=F(\nabla u)$, $F_{i}=F_{\xi_i}(\nabla u)$, $u_i=\frac{\partial u}{\partial x_i}$,
$u_{ij}=\frac{\partial^2 u}{\partial x_i\partial x_j}$ and so on. Denote \begin{equation}
\begin{array}{rcl}\label{add_1}
a_{ij}(\nabla u)(x) &:=&\ds\vs \frac{\partial^2 }{\partial \xi_i \partial \xi_j}\left(\frac12 F^2\right)(\nabla u(x)) =(F_iF_j+FF_{ij})(\nabla u(x)),\\
a_{ijk}(\nabla u)(x)&:=&\ds \frac{\partial^3 }{\partial \xi_i \partial \xi_j \partial \xi_k}\left(\frac12 F^2\right)(\nabla u (x)).
\end{array}
\end{equation}
 In the following we shall write it simply by $a_{ij}$ and $a_{ijk}$ if no confusion appears.
With these notations, we can rewrite the Finsler-Laplacian (\ref{Finsler Laplacian}) as
\begin{eqnarray}\label{Qu}
Qu=a_{ij}u_{ij}.
\end{eqnarray}

For the function $\frac12 F^2(\nabla u)$ we have a Bochner type formula.
\begin{lemma}[Bochner Formula] At a point where $\nabla u\neq 0$, we have
\begin{equation}\label{Bochner}
 a_{ij}\left(\frac12 F^2(\nabla u)\right)_{ij}= a_{ij}a_{kl}u_{ik}u_{jl}+(Qu)_k\frac{\partial }{\partial \xi_k}\left(\frac12 F^2\right)(\nabla u)-a_{ijl}\frac{\partial }{\partial x_l}\left(\frac12 F^2(\nabla u)\right)u_{ij}.
\end{equation}
\end{lemma}

\begin{proof}
The formula is derived from a direct computation.
\begin{eqnarray*}
 a_{ij}(\nabla u)\left(\frac12 F^2(\nabla u)\right)_{ij}&= & a_{ij}\frac{\partial }{\partial x_j}\left(\frac{\partial }{\partial \xi_k}\left(\frac12 F^2\right)(\nabla u)u_{ik}\right)\\
&=& a_{ij}\frac{\partial^2 }{\partial \xi_k \partial \xi_l}\left(\frac12 F^2\right)(\nabla u)u_{ik}u_{jl}+a_{ij}\frac{\partial }{\partial \xi_k}\left(\frac12 F^2\right)(\nabla u)u_{ijk}\\
&=& a_{ij}a_{kl}u_{ik}u_{jl}+\frac{\partial }{\partial \xi_k}\left(\frac12 F^2\right)(\nabla u)\left(\frac{\partial }{\partial x_k}(a_{ij}u_{ij})-(\frac{\partial }{\partial x_k}a_{ij})u_{ij}\right).
\end{eqnarray*}
Taking into account of (\ref{Qu}) and
\begin{eqnarray*}
\frac{\partial }{\partial \xi_k}\left(\frac12 F^2\right)\frac{\partial }{\partial x_k}a_{ij}=a_{ijl}\frac{\partial }{\partial x_l}\left(\frac12 F^2(\nabla u)\right),
\end{eqnarray*}
we get (\ref{Bochner}).

\end{proof}
When $F(\xi)=|\xi|,$  (\ref{Bochner}) is just the usual Bochner formula
\[\frac 12 \Delta(|\nabla u|^2)=|D^2 u|^2+ \langle\nabla u , \nabla (\Delta u)\rangle.\]

We have
a Kato type inequality for the square of ``anisotropic'' norm of Hessian.
\begin{lemma}[Kato inequality]\label{lem1} At a point where $\nabla u\neq 0$, we have
\begin{eqnarray}\label{Kato}
a_{ij}a_{kl}u_{ik}u_{jl}\geq a_{ij}F_kF_lu_{ik}u_{jl}.
\end{eqnarray}
\end{lemma}

\begin{proof}
It is clear that
$$a_{ij}a_{kl}u_{ik}u_{jl}-a_{ij}F_kF_lu_{ik}u_{jl}=a_{ij}FF_{kl}u_{ik}u_{jl}=FF_{i}F_{j}F_{kl}u_{ik}u_{jl}+F^2F_{ij}F_{kl}u_{ik}u_{jl}.$$
Since $(F_{ij})$ is positive definite, we know the first term
\begin{equation*}
FF_{i}F_{j}F_{kl}u_{ik}u_{jl}=FF_{kl}(F_iu_{ik})(F_ju_{jl})\ge 0.
\end{equation*}
The second term $F_{ij}F_{kl}u_{ik}u_{jl}$ is nonnegative as well. Indeed, we can write the matrix $(F_{kl})_{k,l}=O^T\Lambda O$ for some orthogonal matrix $O$ and diagonal matrix $\Lambda=diag(\mu_1,\mu_2,\cdots,\mu_n)$ with $\mu_i\geq0$ for any $i=1,2,\cdots,n$. Set $U=(u_{ij})_{i,j}$ and $\widetilde{U}=OUO^T=(\tilde{u}_{ij})_{i,j}$. Then we have \begin{eqnarray*}
F_{ij}F_{kl}u_{lj}u_{ki}&=& tr(O^T\Lambda OUO^T\Lambda OU)=tr(\Lambda OUO^T\Lambda OUO^T)\\&=& tr(\Lambda \widetilde{U}\Lambda \widetilde{U})=\mu_i\mu_j\tilde{u}_{ij}^2\geq 0,
\end{eqnarray*}
and hence the proof of (\ref{Kato}).
\end{proof}
When $F(\xi)=|\xi|$, (\ref{Kato}) is the usual Kato inequality
\[|\nabla^2 u|^2\ge
|\nabla |\nabla u||^2. \]
The following inequality is crucial to apply the gradient comparison argument in the next Section.

\begin{lemma}\label{basic lem} At a point where $\nabla u\neq 0$, we have
 \begin{equation}\label{eq1}
 a_{ij}a_{kl}u_{ik}u_{jl} \ge \frac{(a_{ij}u_{ij})^2}{n}+\frac{n}{n-1}\left(\frac{a_{ij}u_{ij}}{n}-F_iF_ju_{ij}
\right)^2
\end{equation}
\end{lemma}
\begin{proof}
 Let
\[A=F_iF_ju_{ij} \quad \hbox{ and } B=FF_{ij} u_{ij}.\]
The right hand side of (\ref{eq1}) equals to
\begin{equation*} \label{eq2}\begin{array}{rcl}
&& \ds\vs
\frac{(A+B)^2} n+\frac{n}{n-1}\left(\frac{B}{n}
-\frac{n-1}{n} A\right)^2 =\ds  A^2+\frac{1}{n-1}B^2.
\end{array}
\end{equation*}
The left hand side of (\ref{eq1}) is
\begin{equation*}\label{eq3}
 A^2+2FF_{i}F_{j}F_{kl}u_{ik}u_{jl}+F^2F_{ij}F_{kl}u_{ik}u_{jl}.
\end{equation*}
Since $(F_{ij})$ is semi-positively definite, we know
\begin{equation*}
 \label{eq4}FF_{i}F_{j}F_{kl}u_{ik}u_{jl}=FF_{kl}(F_iu_{ik})(F_ju_{jl})\ge 0.
\end{equation*}
Using the same notations in the proof of Lemma \ref{lem1}, we have
\begin{eqnarray*}
F^2F_{ij}F_{kl}u_{ik}u_{jl}=F^2\mu_i\mu_j\tilde{u}_{ij}^2= F^2\mu_i^2\tilde{u}_{ii}^2+F^2\sum_{i\neq k} \mu_i\mu_k \tilde{u}_{ik}^2\geq F^2\mu_i^2\tilde{u}_{ii}^2,
\end{eqnarray*}
\begin{eqnarray*}
B=FF_{ij}u_{ij}= tr(O^T\Lambda OU)=tr(\Lambda OUO^T)=\mu_i\tilde{u}_{ii}.
\end{eqnarray*}
We claim that  $(F_{ij})$ is a matrix of rank $n-1$, in other words,
one of $\mu_i$ is zero.
Firstly, $F_{ij}u_j=0$. Secondly, for any nonzero $V\perp F_\xi(\nabla u)$, $F_{ij}V^iV^j=a_{ij}V^iV^j>0$. The claim follows easily.
 Thus the H\"older
inequality gives
\[
 F^2\mu_i^2\tilde{u}_{ii}^2\ge \frac {1}{n-1} F^2(\mu_i\tilde{u}_{ii})^2= \frac {1}{n-1} B^2.
\]
Altogether we complete the proof of the Lemma.
\end{proof}

When $F(\xi)=|\xi|$, then (\ref{eq1}) is
\begin{equation*}\label{x1}
|\nabla^2 u|^2\ge \frac {(\Delta u)^2}n+\frac n{n-1}\left(\frac {\Delta u} n-\frac{u_iu_ju_{ij}}{|\nabla u|^2}\right)^2.\end{equation*}


\

We now recall the definition of $F$-mean curvature.
Let $\Omega\subset\mathbb{R}^n$ be a smooth bounded domain, whose boundary $\partial\Omega$ is a
$(n-1)$-dimensional, oriented, compact submanifold without boundary in $\mathbb{R}^n$ .
 We denote by $\nu$ and $d\sigma$ the outward normal of $\partial\Omega$ and area element respectively.  Let
$\{e_\alpha\}_{\alpha=1}^{n-1}$ be a basis of the tangent space
$T_p (\partial\Omega)$ and
$g_{\alpha\beta}=g(e_\alpha,e_\beta)$ and $h_{\alpha\beta}$ be the first and second fundamental form respectively. $\partial \Omega$ is called weakly convex, if $(h_{\alpha \beta})$ is nonnegative definite.
Moreover let $(g^{\alpha\beta})$ be the inverse matrix
of $(g_{\alpha\beta})$ and $\overline{\nabla}$  the covariant
derivative in $\mathbb{R}^n$. The $F$-second fundamental form $h_{\alpha\beta}^F$ and  $F$-mean curvature $H_F$ are defined by
\begin{eqnarray*}
h_{\alpha\beta}^F:=\langle F_{\xi\xi}\circ\overline{\nabla}_{e_\alpha}\nu,e_\beta\rangle
\end{eqnarray*}
and
\begin{eqnarray*}
H_F=\sum_{\alpha,\beta=1}^{n-1} g^{\alpha\beta}h_{\alpha\beta}^F
\end{eqnarray*}
respectively.
$\overrightarrow{H_F}=-H_F\nu$ are called $F$-mean curvature vector (it is easy to check that all definitions are independent of the choice of coordinate).
$\partial \Omega$ is called weakly $F$-convex (F-mean convex, resp.) if
$(h^F_{\alpha\beta})$ is nonnegative definite ($H_F\ge 0$ resp.).
It is well known that when we consider a variation of $\partial\Omega$ with
variation vector field $\varphi\in C_0^\infty(\partial\Omega,\mathbb{R}^n)$,
 the first
variation of the $F$-area functional
$\mathcal{F}(X):=\int_{\partial\Omega} F(\nu) d\sigma$ reads as
$$\delta_\varphi\mathcal{F}(X)=-\int_{\partial\Omega}
\langle \overrightarrow{H_F},\varphi \rangle d\sigma.$$
It is easy to see from the convexity of $F$ that $h_{\alpha\beta}^F$ being nonnegative definite is equivalent that the ordinary
second fundamental  form $h_{\alpha\beta}$ being  nonnegative definite, in other words, there is no difference between weakly $F$-convex and weakly convex.
However, $F$-mean convex is different from mean convex.
For more properties of $H_F$, we refer to \cite{WX1} and reference therein. Here we will use the following lemma in \cite{WX1}, which interprets  the relation between Finsler-Laplacian and $F$-mean curvature of level sets of functions.
\begin{lemma}[\cite{WX1}, Theorem 3]\label{F mean curvature}
Let $u$ be a $C^2$ function with a regular level set $S_t:=\{x\in\overline{\Omega}|u=t\}$.
Let $H_F(S_t)$ be the $F$-mean
curvature of the level set $S_t$.  We then have
\begin{equation*}
 Qu(x)=-F H_F(S_t)+F_iF_j u_{ij}= -F H_F(S_t)+\frac{\partial^2 u}{\partial \nu_F^2}\end{equation*}
for $x$ with $u(x)=t$, where $\nu_F:=F_{\xi}(\nu) =-F_\xi
(\nabla u)$.
\end{lemma}
We point out that  we have used the inward normal in \cite{WX1} and there is an sign error in the formula (5) there. Hence the term $FH_F(S_t)$ in the formula (9) there should be read as $-FH_F(S_t)$.

\

\section{Sharp estimate of the first Neumann  eigenvalue}\label{NFG}

It is well-known that the existence of Neumann first eigenfunction can be obtained from the direct method in the calculus of variations. We note that the first Neumann  eigenfunction must change sign, for its average vanishes.

In this Section we first prove the following gradient comparison theorem, which is the most crucial part for the proof of the sharp estimate. For simplicity, we write $\lambda_1$ instead of $\lambda_1^N$ through this section.

\begin{theorem} \label{Comparison thm}
Let $\Omega,u,\lambda_1$  be as in Theorem \ref{main thm2}. Let $v$ be a solution of the 1-D model problem on some interval $(a,b)$:

\begin{equation}\label{1-D model}
v''-Tv'=-\lambda_1 v,\quad v'(a)=v'(b)=0,\quad v'>0,
\end{equation}
with $T(t)=-\frac{n-1}{t} \hbox{ or } 0.$
Assume that $[\min u,\max u]\subset [\min v,\max v]$,
then
\begin{equation}\label{grad comparison}
F(\nabla u)(x)\leq v'( v^{-1}(u(x))).
\end{equation}

\end{theorem}

\begin{proof} First, since $\int u=0$, we know that $\min u<0$ while $\max u>0$. We may assume that $[\min u,\max u]\subset (\min v,\max v)$ by multiplying $u$ by a constant $0<c<1$. If we prove the result for this $u$, then letting $c\to 1$ we have
 (\ref{grad comparison}).

Under the condition $[\min u,\max u]\subset (\min v,\max v)$, $v^{-1}$ is smooth on a neighborhood $U$ of $[\min u,\max u]$.

Consider  $P:=\psi(u)(\frac{1}{2}F(\nabla u)^2-\phi(u)),$ where $\psi,\phi\in C^\infty (U)$ are two positive smooth functions to be determined later.
We first assume that $P$ attains its maximum at $x_0\in\Omega$, and then we will
consider the case that $x_0\in\partial\Omega$. If $ \nabla u(x_0)=0$, $P\leq 0$ is obvious. Hence we assume $\nabla u(x_0)\neq 0$.
From now on we compute at $x_0$. As in Section \ref{Pre}, we use the notation (\ref{add_1}). Since $x_0$ is the maximum of $P$,
we have that
\begin{equation} \label{gradient N}  P_i(x_0)=0,\end{equation}
\begin{equation}\label{second derivative N}a_{ij}(x_0) P_{ij}(x_0)\leq 0.\end{equation}
Equality in (\ref{gradient N}) gives
\begin{eqnarray}\label{gradient N1}
\frac{\partial}{\partial x_i}\left(\frac12 F^2(\nabla u)-\phi(u)\right)=-\frac{\psi(u))_i}{\psi^2}P,\quad F_iF_ju_{ij}=\phi'-\frac{\psi'}{\psi^2}P.
\end{eqnarray}
Then we compute $a_{ij}P_{ij}$.
\begin{eqnarray*}
a_{ij}P_{ij}&=&\frac{P}{\psi}a_{ij}(\psi(u))_{ij}+\psi a_{ij}\frac{\partial}{\partial x_ix_j}\left(\frac12 F^2(\nabla u)-(\phi(u))\right)\\
&&+ 2a_{ij}(\psi(u))_i\frac{\partial}{\partial x_j}\left(\frac12 F^2(\nabla u)-\phi(u)\right).
\end{eqnarray*}
It is easy to see from Proposition \ref{property} that \begin{equation}\label{homogeneity}\frac{\partial}{\partial \xi_i}\left(\frac12 F^2\right)(\nabla u)u_i= F^2(\nabla u),\quad a_{ij}u_iu_j=F^2(\nabla u),\quad a_{ijk}u_k=0.\end{equation}
By using (\ref{gradient N1}), (\ref{homogeneity}), the Bochner formula (\ref{Bochner}) and eigenvalue equation (\ref{first eigenvalue}), we get
\begin{eqnarray}\label{second derivative N1} a_{ij}P_{ij}&=&(-\lambda_1 u \frac{\psi'}{\psi}+F^2\frac{\psi''}{\psi}-2F^2\frac{{\psi'}^2}{\psi^2})P\nonumber\\&&+\psi ( a_{ij}a_{kl}u_{ik}u_{jl}-\lambda_1 F^2)+\psi(\lambda_1 u\phi'-F^2\phi'').\end{eqnarray}
Applying Lemma \ref{basic lem} to (\ref{second derivative N1}), replacing $F^2$ by $2\frac{P}{\psi}+\phi$ and using (\ref{gradient N1}), (\ref{first eigenvalue}), (\ref{second derivative N}), we deduce

\begin{equation}\label{second derivative N2}
\begin{array}{rcl}
 0\geq a_{ij}P_{ij}&\geq &  \ds\vs (-\lambda_1 u \frac{\psi'}{\psi}+F^2\frac{\psi''}{\psi}-2F^2\frac{{\psi'}^2}{\psi^2})P+\psi(\lambda_1 u\phi'-F^2\phi'')\\
 && \ds\vs +\psi \left(\frac{(a_{ij}u_{ij})^2}{n}+\frac{n}{n-1}\left(\frac{a_{ij}u_{ij}}{n}-F_iF_ju_{ij}
\right)^2-\lambda_1 F^2\right)\\ &=& \ds\vs \frac{1}{\psi}\left[2\frac{\psi''}{\psi}-(4-\frac{n}{n-1})\frac{\psi'^2}{\psi^2}\right]P^2\\
&&\ds\vs +\left[2\phi\left(\frac{\psi''}{\psi}-2\frac{\psi'^2}{\psi^2}\right)-\frac{n+1}{n-1}\frac{\psi'}{\psi}\lambda_1 u-\frac{2n}{n-1}\frac{\psi'}{\psi}\phi'-2\lambda_1-2\phi''\right]P\\
&&\ds\vs +\psi\left[\frac{1}{n-1}\lambda_1^2u^2+\frac{n+1}{n-1}\lambda_1u \phi'+\frac{n}{n-1}\phi'^2-2\lambda_1\phi-2\phi\phi''\right]\\
&:=& \ds a_1 P^2+a_2P+a_3.\end{array}\end{equation}

We are lucky to observe that the coefficients $a_i,$ $i=1,2,3$, coincide with  those appearing in the
ordinary Laplacian case (see e.g. \cite{BQ}, Lemma 1).
The
next step is to choose suitable positive functions $\psi$ and $\phi$ such that $a_1,a_2>0$ everywhere and $a_3=0$, which had already be done in \cite{BQ}. For completeness, we sketch the main idea here.

Choose $\phi(u)=\frac{1}{2}v'(v^{-1}(u))^2$, where $v$ is a solution of 1-D problem (\ref{1-D model}).
One can compute that $$\phi'(u)=v''(v^{-1}(u)),\phi''(u)=\frac{v'''}{v'}(v^{-1}(u)).$$
Setting $t=v^{-1}(u)$ and $u=v(t)$ we have
\begin{eqnarray*}\frac{a_3(t)}{\psi}&=&\frac{1}{n-1}\lambda_1^2v^2+\frac{n+1}{n-1}\lambda_1 v v''+\frac{n}{n-1}v''^2-\lambda_1v'^2-v'v'''
\\&=&-v'(v''-Tv'+\lambda_1v)'+\frac{1}{n-1}(v''-Tv'+\lambda_1v)(nv''+Tv'+\lambda_1v)=0.\end{eqnarray*}
Here we have used that $T$ satisfies $T'=\frac{T^2}{n-1}.$
For $a_1,a_2$, we introduce $$X(t)=\lambda_1\frac{v(t)}{v'(t)},\quad\psi (u) =\exp(\int h(v(t))), \quad f(t)=-h(v(t))v'(t).$$
With these notations,
we have
\begin{eqnarray*}
f'=-h'v'^2+f(T-X),\end{eqnarray*}
\begin{eqnarray*}
v'|_{v^{-1}}^2 a_1\psi=2f(T-X)-\frac{n-2}{n-1}f^2-2f':=2(Q_1(f)-f'),\end{eqnarray*}
\begin{eqnarray*}
 a_2=f(\frac{3n-1}{n-1}T-2X)-2T(\frac{n}{n-1}T-X)-f^2-f':=Q_2(f)-f'.\end{eqnarray*}
 We may now use Corollary 3 in \cite{BQ}, which says that there exists a bounded function $f$ on $[\min u,\max u]\subset (\min v,\max v)$ such that $f'<\min\{Q_1(f),Q_2(f)\}$.

In view of (\ref{second derivative N2}), we know that by our choice of $\psi$ and $\phi$,
 $P(x_0)\leq 0$, and hence $P(x)\leq 0$ for every $x\in\Omega$, which leads to (\ref{grad comparison}).

\

Now we consider the case $x_0\in \partial \Omega$.
 Suppose that $P$ attains its maximum at $x_0\in\partial\Omega$.
Consider a new vector field  $V(x)=(V^i(x))_{i=1}^n$ defined on $\partial \Omega$ by $$V^i(x)=\sum_{j=1}^n a_{ij}(\nabla u(x))\nu^j(x).$$ Thanks to the positivity of  $a_{ij}$, $V(x)$ must point outward. Hence $\frac{\partial P}{\partial V}(x_0)\geq 0$.

On the other hand, we see from the Neumann boundary condition and homogeneity of $F$ that $$\frac{\partial u}{\partial V}(x_0)=u_i a_{ij}(\nabla u(x))\nu^j=FF_j\nu^j=0 .$$
Thus we have \begin{eqnarray}\label{normal derivative} 0\leq\frac{\partial P}{\partial V}(x_0)=\psi FF_iu_{ij} a_{jk} \nu^k.\end{eqnarray}

Choose now local coordinate $\{e_i\}_{i=1,\cdots,n}$ around $x_0$ such that $e_n=\nu$ and $\{e_\alpha\}_{\alpha=1,\cdots,n-1}$ is the orthonormal basis of tangent space of $\partial \Omega$. Denote by $h_{\alpha\beta}$ the second fundamental form of $\partial\Omega$. By the assumption that $\partial\Omega$ is weakly convex, we know the matrix $(h_{\alpha\beta})\geq 0$. 

The Neumann boundary condition implies
\begin{eqnarray}\label{F_n} F_i\nu^i(x_0)=F_n(x_0)=0.\end{eqnarray}
 By taking tangential derivative of (\ref{F_n}), we have
 \begin{equation*}
 \quad D_{e_\beta}(\sum_{i=1}^n F_i\nu^i)(x_0)=0,
 \end{equation*} for any $\beta=1,\cdots,n-1.$
Computing  $D_{e_\beta}(\sum_{i=1}^n F_i\nu^i)(x_0)$ explicitly, we have \begin{eqnarray}\label{tangent derivative}
0=D_{e_\beta}(\sum_{i=1}^n F_i\nu^i)(x_0)&=&\sum_{i,j=1}^n F_{ij}u_{j\beta}\nu^i+\sum_{i=1}^n F_i\nu_\beta^i\nonumber\\
&=&\sum_{i,j=1}^n F_{ij}u_{j\beta}\nu^i+\sum_{i=1}^n \sum_{\gamma=1}^{n-1}F_ih_{\beta\gamma}e_\gamma^i\nonumber\\
&=&\sum_{j=1}^n F_{nj}u_{j\beta}+\sum_{\gamma=1}^{n-1} F_\gamma h_{\beta\gamma}.\end{eqnarray}
In the last equality we have used $\nu_n=1$ and $\nu_\beta=0$ for $\beta=1,\cdots,n-1$ in the chosen coordinate.

Combining (\ref{normal derivative}), (\ref{F_n}) and (\ref{tangent derivative}), we obtain
\begin{eqnarray*}0\leq \frac{\partial P}{\partial V}(x_0)&=&\sum_{i,j,k=1}^{n}\psi FF_iu_{ij} a_{jk} \nu^k =\psi F\sum_{\alpha=1}^{n-1}\sum_{j=1}^n F_\alpha u_{\alpha j}a_{jn}\nonumber\\
&=&\psi F\sum_{\alpha=1}^{n-1}\sum_{j=1}^n F_\alpha u_{\alpha j}F_{jn}= -\psi F\sum_{\alpha,\gamma=1}^{n-1} F_\alpha F_\gamma h_{\alpha\gamma}\leq 0.\end{eqnarray*}



 Therefore we obtain that $\frac{\partial P}{\partial V}(x_0)=0$. Since the tangent derivatives of $P$ also vanishes, we have $\nabla P(x_0)=0$. It's also the case that (\ref{second derivative N}) holds. Thus the previous proof for an interior maximum also works  in this case. This finishes the proof of Theorem \ref{Comparison thm}.
\end{proof}

\

Following the idea of \cite{BQ}, besides the gradient comparison with the 1-D models, in order to prove the sharp estimate on the first eigenvalue of the Finsler-Laplacian,  we need to study many properties of the 1-D models, such as the difference $\delta(a)=b(a)-a$ as a function of $a\in [0,+\infty]$, where $b(a)$ is the first number that $v'(b(a))=0$ (Note that  $v'> 0\hbox{ in }(a,b(a))$). As we already saw in Theorem \ref{Comparison thm}, the 1-D model (\ref{1-D model}) appears the same as that in the Laplacian case. Therefore, we can use directly the results of \cite{BQ} on the properties of the 1-D models. Here we use some simpler statement from \cite{Va}.

We define $\delta(a)$ as a function of $a\in [0,+\infty]$ as follows. On one hand, we denote $\delta(\infty)=\frac{\pi}{\sqrt{\lambda_1}}$. This number comes from the 1-D model (\ref{1-D model}) with $T=0$. In fact, it is easy to see that solutions of  the 1-D model (\ref{1-D model}) with $T=0$ can be explicitly written as $$v(t)=\sin{\sqrt{\lambda_1}t}$$ up to dilations. Hence in this case, $b(a)-a=\frac{\pi}{\sqrt{\lambda_1}}$ for any $a\in \mathbb{R}$. On the other hand,  we denote $\delta(a)=b(a)-a$ as a function of $a\in [0,+\infty)$ relative to  the 1-D model (\ref{1-D model}) with $T=-\frac{n-1}{x}$.

\

The following property of $\delta(a)$ was proved in  \cite{BQ,Va}.

\begin{lemma}[\cite{BQ} or \cite{Va}, Th. 5.3, Cor. 5.4] \label{thm1}
The function $\delta(a): [0,\infty]\to \mathbb{R}^+ $is a continuous function such that
\begin{eqnarray*}\delta(a)>\frac{\pi}{\sqrt{\lambda_1}},\end{eqnarray*}
\begin{eqnarray*}\delta(\infty)=\frac{\pi}{\sqrt{\lambda_1}}.\end{eqnarray*}
$m(a):=v(b(a))<1$, $\lim_{a\to \infty} m(a)=1$ and $m(a)=1$ if and only if $a=\infty$.
\end{lemma}

In order to prove the main result, we also need the following comparison theorem on the maximum values of eigenfunctions. This theorem is obtained as a consequence of a standard property of the volume of small balls with respect to some invariant measure (see \cite{BQ}, Section 6).

\begin{lemma}
\label{thm2}
Let $\Omega,u,\lambda_1$  be as in Theorem \ref{main thm2}. Let $v$ be a solution of the 1-D model problem on some interval $(0,\infty)$:

\begin{equation*}
v''=-\frac{n-1}{t}v'-\lambda_1 v,\quad v(0)=-1,\quad v'(0)=0.
\end{equation*}
Let $b$ be the first number after $0$ with $v'(b)=0$  and denote $m=v(b)$. Then $\max u\geq m.$

\end{lemma}

The proof of Lemma \ref{thm2} is
 similar to that of \cite{BQ}, Th. 11. The essential part is the gradient comparison Theorem \ref{Comparison thm}. We omit it here.

\

Now we are in position to prove  Theorem \ref{main thm2}.

\

\noindent{\it Proof of Theorem \ref{main thm2}.}
Let $u$ be an eigenfunction with eigenvalue $\lambda_1$. Since $\int u=0$, we may assume $\min u=-1$ and $0\leq k= \max u\leq 1.$
Given a solution $v$ to (\ref{1-D model}), denote $m(a)=v(b(a))$ with $b(a)$ the first number with $v'(b(a))=0$ after $a$.

Lemma \ref{thm1} and \ref{thm2} imply that for any eigenfunction $u$, there exists a solution $v$ to (\ref{1-D model}) such that
$\min v=\min u=-1$ and $\max v=\max u=k\leq 1$.

We now get the expected estimate by using Theorem \ref{Comparison thm}.
Choosing $x_1,x_2\in \overline{\Omega}$ with $u(x_1)=\min u=-1, u(x_2)=\max u=k$ and $\gamma(t):[0,1]\to \overline{\Omega}$ the minimal geodesic from $x_1$ to $x_2$. Consider the subset $I$ of [0,1] such that $\frac{d}{dt}u(\gamma(t))\geq 0$. By the gradient comparison estimate (\ref{grad comparison}) and Lemma \ref{thm1}, we have
\begin{eqnarray*}d_F&\geq &  \int_{0}^1 F^0(\dot{\gamma}(t))dt\geq \int_I F^0(\dot{\gamma}(t))dt \\&\geq &\int_{0}^1 \frac{1}{F(\nabla u)}\langle\nabla u, \dot{\gamma}(t)\rangle
dt=\int_{-1}^k \frac{1}{F(\nabla u)}du\\&\geq &\int_{-1}^k \frac{1}{v'(v^{-1}(u))}du=\int_{a}^{b(a)} dt
=\delta(a)\geq \frac{\pi}{\sqrt{\lambda_1}},\end{eqnarray*}
which leads to $$\lambda_1\geq\frac{\pi^2}{d_F^2}.$$

\

We are remained to prove the equality case. The idea of proof follows from \cite{HW}.
Here we need to pay more attention on the points with vanishing  $\nabla u$.

Assume that $\lambda_1=\frac{\pi^2}{d_F^2}$. It can be easily seen from the proof of  Theorem \ref{main thm2} that $a=\infty$, which leads to $\max u=\max v=1$ by Lemma \ref{thm1}.   We will prove that $\Omega$ is in fact a segment in $\mathbb{R}$.
We divide the proof into several steps.

\

\noindent{\bf Step 1:}  $S:=\{x\in\overline{\Omega}|u(x)=\pm 1\}\subset \partial\Omega$.

Let $\mathcal{P}=F(\nabla u)^2+\lambda_1u^2$. After a simple calculation by using Bochner formula (\ref{Bochner}) and Kato inequality (\ref{Kato}), we obtain
\begin{eqnarray*}
\frac12 a_{ij}\mathcal{P}_{ij}&=& a_{ij}a_{kl}u_{ik}u_{jl}-\frac12 a_{ijl}u_{ij} \mathcal{P}_l-\lambda_1^2u^2\\&\geq & a_{ij}F_kF_l u_{ik}u_{jl}-\frac12 a_{ijl}u_{ij} \mathcal{P}_l-\lambda_1^2u^2\\&= &
-\frac12 a_{ijl}u_{ij} \mathcal{P}_l+\frac{1}{4F^2}(a_{ij}\mathcal{P}_i\mathcal{P}_j-4\lambda_1 uu_i\mathcal{P}_i) \hbox{ on }\Omega\setminus \mathcal{C}.
\end{eqnarray*}
Namely,
\begin{eqnarray}\label{eqn1}
\frac12 a_{ij}\mathcal{P}_{ij}+b_i\mathcal{P}_i\geq 0 \hbox{ on }\Omega\setminus \mathcal{C}
\end{eqnarray}
for some $b_i\in C^0(\Omega)$.
If $\mathcal{P}$ attains its maximum on $x_0\in\partial\Omega$, then arguing as in Theorem \ref{Comparison thm},  we have that $\nabla \mathcal{P}(x_0)=0$. However, from the Hopf Theorem, $\nabla \mathcal{P}(x_0)\neq 0$, a contradiction. Hence  $\mathcal{P}$ attains its maximum at $\mathcal{C}$, and therefore,
\begin{eqnarray}\label{eqn2}
\mathcal{P}\leq \lambda_1.
\end{eqnarray}
Take any two points $x_1,x_2\in S$ with $u(x_1)=-1, u(x_2)=1$. Let $$\gamma(t)=\left(1-\frac{t}{F^0(x_2-x_1)}\right)x_1+\frac{t}{F^0(x_2-x_1)}x_2 :[0,l]\to \overline{\Omega}$$ be the straight line from $x_1$ to $x_2$, where $l:=F^0(x_2-x_1)$ is the distance from  $x_1$ to $x_2$ with respect to $F$. Denote $f(t):=u(\gamma(t))$. It is easy to see $F^0(\dot{\gamma}(t))=1$. It follows from (\ref{eqn2}) and Cauchy-Schwarz inequality (\ref{CS}) that
\begin{eqnarray}\label{eqn3}
|f'(t)|=|\nabla u(\gamma(t))\cdot \dot{\gamma}(t)|\leq F(\nabla u)(\gamma(t))\leq \sqrt{\lambda_1(1-f(t)^2)}.
\end{eqnarray}
Here we have used the Cauchy-Schwarz inequality (\ref{CS}) again.
Hence
\begin{eqnarray}\label{eqn4}
d_F\geq l &\geq &\int_{\{0\leq t\leq l,\, f'(t)>0\}} dt\geq \int_0^l\frac{1}{\sqrt{\lambda_1}}\frac{f'(t)}{\sqrt{1-f(t)^2}} dt\nonumber\\&=&\frac{1}{\sqrt{\lambda_1}}\int_{-1}^1\frac{1}{\sqrt{1-x^2}}dx=\frac{\pi}{\sqrt{\lambda_1}}.
\end{eqnarray}
Since $d_F=\frac{\pi}{\sqrt{\lambda_1}}$, we must have $d_F= l$, which means $S\subset \partial\Omega$.

\

\noindent{\bf Step 2:}  $\mathcal{P}=F^2(\nabla u)+\lambda_1 u^2\equiv\lambda_1$ in $\overline{\Omega}$, hence $S\equiv \mathcal{C}$.

Indeed, from Step 1, we know that $\Omega^*:=\overline{\Omega}\setminus S$ is connected. Let $E:=\{x\in\Omega^*:\mathcal{P}=\lambda_1\}$. It is clear that $E$ is closed.
In view of (\ref{eqn1}), thanks to the strong maximum principle we know that  $E$ is also open. we now show that $E$ is nonempty. Indeed, from the fact that all inequalities  in (\ref{eqn3}) and (\ref{eqn4}) are equality, we obtain $f(t)=u(\gamma(t))=-\cos \sqrt{\lambda_1}t$ for $t\in(0,l)$.
Hence  $$\mathcal{P}(\gamma(t))= f'(t)^2+\lambda_1f(t)^2=\lambda_1.$$ Thus $E$ is nonempty, open, closed in $\Omega^*$. Therefore, we obtain $\mathcal{P}\equiv\lambda_1$ in $\overline{\Omega}$ (for $x\in S$, $\mathcal{P}=\lambda_1$ is obvious).

\

\noindent{\bf Step 3:} Define $X=\frac{\nabla u}{F(\nabla u)}$ in $\Omega^*$ and $X^*$ the cotangent vector given by $X^*(Y)=\langle X,Y\rangle$ for any tangent vector $Y$. Then in $\Omega^*$, we claim that
\begin{eqnarray}\label{eqn5}
D^2u=-\lambda_1 u X^*\otimes X^*,
\end{eqnarray}
and moreover $X=\overrightarrow c$ for some constant vector $\overrightarrow c$.

First, taking derivative of  $F^2(\nabla u)+\lambda_1 u^2\equiv\lambda_1$ gives
\begin{eqnarray}\label{eqn8}
F_iF_ju_{ij}=-\lambda_1 u.
\end{eqnarray}
On the other hand, since $\mathcal{P}\equiv \lambda_1$, the proof of (\ref{eqn1}) leads to
\begin{eqnarray}\label{eqn6}
a_{ij}a_{kl}u_{ik}u_{jl}=\lambda_1^2 u^2= (F_iF_ju_{ij})^2.
\end{eqnarray}
(\ref{eqn6}) in fact gives that
\begin{eqnarray}\label{eqn7}
F_{ij}F_{kl}u_{ik}u_{jl}=0.
\end{eqnarray}
Set $X^\perp:=\{V\in \mathbb{R}^n|V\perp X\}.$ $X^\perp$ is an $(n-1)$-dim vector subspace. Note that
$(F_{ij})$ is exactly matrix of rank $n-1$ (see the proof of Lemma \ref{basic lem}) and $F_{ij}X^j=0$. It follows from this fact and (\ref{eqn7}) that
\begin{eqnarray}\label{eqn9}
u_{ij}V^iV^j=0 \hbox{ for any }V\in X^\perp.
\end{eqnarray}
(\ref{eqn8}) and (\ref{eqn9}) imply (\ref{eqn5}), which in turn
implies \begin{eqnarray}\label{eqn10}
u_{ij}=\frac{-\lambda_1uu_iu_j}{F^2(\nabla u)}.
\end{eqnarray}
By differentiating $X$, we obtain from (\ref{eqn10}) that
\begin{eqnarray*}
\nabla_i X^j=\frac{u_{ij}}{F(\nabla u)}-\frac{u_j}{F^2(\nabla u)}F_ku_{ki}=0.
\end{eqnarray*}
Thus $X=\overrightarrow{c}$ in $\Omega^*$.

\

\noindent{\bf Step 4:} The maximum point and the minimum point are unique.

We already knew that $f(t)=u(\gamma(t))=-\cos\sqrt{\lambda_1}t$ and  $\nabla u(\gamma(t))\neq 0$ for $t\in(0,l)$. Hence $u$ is $C^2$ along $\gamma(t)$ for $t\in(0,l)$ and it follows that
\begin{eqnarray}\label{eqn11}
D^2u\left(\dot{\gamma}(t),\dot{\gamma}(t)\right)\bigg|_{\gamma(t)}=\lambda_1 \cos t \hbox{ for any }t\in(0,l).
\end{eqnarray}
On the other hand,  we deduce from (\ref{eqn5}) that
\begin{eqnarray}\label{eqn12}
D^2u\left(\dot{\gamma}(t),\dot{\gamma}(t)\right)\bigg|_{\gamma(t)}=-\lambda_1 u(\gamma(t))\langle X,\dot{\gamma}(t) \rangle^2.
\end{eqnarray}
Combining (\ref{eqn11}) and (\ref{eqn12}), taking $t\to 0$, we get \begin{eqnarray*}
|\langle X,\dot{\gamma}(t) \rangle |=1=F(X)F^0(\dot{\gamma}(t)),
\end{eqnarray*}
which means equality in Cauchy-Schwarz inequality (\ref{CS}) holds. Hence $X=\pm F_\xi^0(\dot{\gamma}(t))$.
Noting that $\dot{\gamma}(t)=\frac{x_2-x_1}{F^0(x_2-x_1)}$, we have $$X=F_\xi^0(x_2-x_1).$$
Suppose there is some point $x_3$ with $u(x_3)=1$, using the same argument, we obtain $X=F_\xi^0(x_3-x_1).$ In view of $F^0(x_3-x_1)=F^0(x_2-x_1)$, we conclude $x_3=x_2$. Therefore, there is only one maximum point as well as one minimum point.

\

\noindent{\bf Step 5:}  $\Omega$ is a segment in $\mathbb{R}$.

Suppose $\Omega\subset\mathbb{R}^n$ for $n\geq 2$. We see from Step 4 that for most of points of $\partial\Omega$, $\nabla u\neq 0$, and at these points  $X=\frac{\nabla u}{F(\nabla u)}$ lies in the tangent spaces due to the Neumann boundary condition, which is impossible because $X$ is a constant vector, a contradiction. We complete the proof.\qed

\

\section{Estimate of the first Dirichlet eigenvalue}\label{DFG}

As in Section \ref{NFG}, for simplicity, we write $\lambda_1$ instead of $\lambda_1^D$ through this section.

It is well-known that the existence of first Dirichlet  eigenfunction can be easily proved by using the direct method in the calculus of variations. Moreover, by the assumption that $F$ is even,
the first Dirichlet  eigenfunction $u$  does not change sign (see \cite{BFK}, Th. 3.1). We may assume $u$ is non-negative.
By multiplying $u$ by a constant, we can also assume that $\sup_\Omega u=1$ and $\inf_\Omega u=0$  without loss of generality.

For any $\alpha,\beta\in \mathbb{R}$ with $\alpha> 0, \beta^2>\sup (\alpha+ u)^2$, consider function $$P(x)=\frac{F^2(\nabla u)}{2(\beta^2-(\alpha+u)^2)}.$$ Suppose that $P(x)$ attains its maximum at $x_0\in \overline{\Omega}$.

With the assumption that $\Omega$ is $F$-mean convex, we first exclude the possibility $x_0\in \partial \Omega$ with $\nabla u(x_0)\neq 0$.
Indeed, suppose we have $x_0\in\partial \Omega$ with $\nabla u(x_0)\neq 0$. Define $\nu_F:=F_\xi(\nu)$  on $\partial \Omega=\{x\in\overline{\Omega}|u(x)=0\}.$ In view of  $\langle\nu_F,\nu\rangle=F(\nu)>0$, $\nu_F$ must point outward. From the Dirichlet boundary condition, we know $\nu=-\frac{\nabla u}{|\nabla u|}$ for $\nabla u\neq 0$. Hence $\nu_F=-F_\xi(\nabla u)$.
Since $P$ attains maximum at $x_0$, we have
\begin{eqnarray*}0\leq\frac{\partial P}{\partial \nu_F}(x_0)&=&\frac{FF_iu_{ij}\nu_F^j}{\beta^2-(\alpha+u)^2}+F^2\frac{\alpha \frac{\partial u}{\partial\nu_F}}{(\beta^2-(\alpha+u)^2)^2}\end{eqnarray*}
Hence \begin{eqnarray*}-\frac{\partial^2 u}{\partial \nu_F^2}+\frac{F\alpha\frac{\partial u}{\partial\nu_F}}{\beta^2-\alpha^2}\geq 0.\end{eqnarray*}
Note that $\frac{\partial u}{\partial\nu_F}=-F(\nabla u)$. Since $\partial \Omega$  itself is  a level set of $u$, we can apply Lemma \ref{F mean curvature} to obtain
$$\frac{\partial^2 u}{\partial \nu_F^2}=Qu+FH_F.$$ In view of $Qu(x_0)=-\lambda_1 u(x_0)=0$, we obtain that
$$-FH_F-F^2\frac{\alpha}{\beta^2-\alpha^2}\geq 0.$$
This contradicts the fact that $H_F(\partial\Omega)\geq 0$.

On the other hand, if  $\nabla u(x_0)=0$, then $F(\nabla u)(x_0)=0$ and $P(x_0)=0$ which implies
$F(\nabla u)=0$, i.e., $u$ is constant, a contradiction.

Therefore we may assume $x_0\in\Omega$ and $\nabla u(x_0)\neq 0$.
 Since $a_{ij}$ is positively definite on $\overline{\Omega}\setminus\mathcal{C}$, where $\mathcal{C}:=\{x|\nabla u(x)=0\}$, it follows from the maximum principle that
\begin{equation} \label{gradient}P_i(x_0)=0,\end{equation}
\begin{equation}\label{second derivative}a_{ij}(x_0) P_{ij}(x_0)\leq 0.\end{equation}
From now on we will compute at the point $x_0$.
Equality (\ref{gradient}) gives
\begin{equation} \label{gradient1}\frac{\partial}{\partial x_i}\left(\frac12 F^2(\nabla u)\right)=-\frac{F^2(\nabla u) (\alpha+u) u_i}{\beta^2-(\alpha+u)^2}.\end{equation}
Then we compute  $a_{ij}(x_0) P_{ij}(x_0)$.
\begin{eqnarray*}a_{ij}(x_0) P_{ij}(x_0)=&&\frac{1}{\beta^2-(\alpha+u)^2}a_{ij}\frac{\partial^2}{\partial x_i\partial x_j}\left(\frac12 F^2(\nabla u)\right)\\&&+2a_{ij}\frac{\partial}{\partial x_i}\left(\frac12 F^2(\nabla u)\right)\frac{\partial}{\partial x_j}\left(\frac{1}{\beta^2-(\alpha+u)^2}\right)\\&&+a_{ij}\frac{\partial^2}{\partial x_i\partial x_j}\left(\frac{1}{\beta^2-(\alpha+u)^2}\right)\frac12 F^2(\nabla u)\\&=& I+II+III.\end{eqnarray*}

By using (\ref{gradient1}), (\ref{homogeneity}), Bochner formula (\ref{Bochner}) and equation (\ref{first eigenvalue}), we obtain
\begin{eqnarray}\label{I}I=\frac{1}{\beta^2-(\alpha+u)^2}\left[ a_{ij}a_{kl}u_{ik}u_{jl}-\lambda_1F^2\right],\end{eqnarray}
\begin{equation}\label{II}II
=-\frac{4F^4(\alpha+u)^2}{(\beta^2-(\alpha+u)^2)^3},\end{equation}
\begin{eqnarray}\label{III}III
=\frac{F^4}{(\beta^2-(\alpha+u)^2)^2}+\frac{4F^4(\alpha+u)^2}{(\beta^2-(\alpha+u)^2)^3}-\frac{\lambda_1 F^2u(\alpha+u)}{(\beta^2-(\alpha+u)^2)^2}.
\end{eqnarray}
We now apply Lemma \ref{lem1} to (\ref{I}) and obtain
\begin{eqnarray*} a_{ij}a_{kl}u_{ik}u_{jl}&\geq &  a_{ij}F_kF_lu_{ik}u_{jl}\\
&=&\frac{1}{F^2}a_{ij}\frac{\partial }{\partial x_i}\left(\frac12 F^2(\nabla u)\right)\frac{\partial }{\partial x_j}\left(\frac12 F^2(\nabla u)\right)\\&=&\frac{F^4(\alpha+u)^2}{(\beta^2-(\alpha+u)^2)^2}.
\end{eqnarray*}
Here we have used (\ref{gradient1}) and (\ref{homogeneity}) again in the last equality.
Therefore, we have
\begin{eqnarray}\label{I1}I\geq \frac{F^4(\alpha+u)^2}{(\beta^2-(\alpha+u)^2)^3}-\frac{\lambda_1F^2}{\beta^2-(\alpha+u)^2}.\end{eqnarray}
Combining (\ref{second derivative}), (\ref{II}), (\ref{III}) and (\ref{I1}), we obtain
\begin{eqnarray*}0\geq a_{ij}P_{ij}\geq \frac{F^4\beta^2}{(\beta^2-(\alpha+u)^2)^3}-\frac{\lambda_1F^2}{\beta^2-(\alpha+u)^2}-\frac{\lambda_1 F^2u(\alpha+u)}{(\beta^2-(\alpha+u)^2)^2}.\end{eqnarray*}
It follows that
\begin{eqnarray}\label{grad est}\frac{F^2(\nabla u)}{\beta^2-(\alpha+u)^2}(x_0)\leq \frac{\lambda_1}{\beta^2}(\beta^2-\alpha(\alpha+u)).\end{eqnarray}

Noting that $\sup_\Omega u=1$ we  choose $\alpha>0$ and $\beta=\alpha+1$.  Then estimate (\ref{grad est}) becomes
\begin{eqnarray*}\frac{F^2(\nabla u)}{(\alpha+1)^2-(\alpha+u)^2}(x_0)\leq \lambda_1\left(1-\frac{\alpha(\alpha+u)}{(\alpha+1)^2}\right)\leq \lambda_1.\end{eqnarray*}
Hence we conclude, for any $x\in\overline{\Omega}$,
\begin{eqnarray}\label{grad est2}\frac{F^2(\nabla u)}{(\alpha+1)^2-(\alpha+u)^2}\leq \lambda_1.\end{eqnarray}

Choose $x_1\in\Omega$ with $u(x_1)=\sup u=1$ and $x_2\in \partial\Omega$ with $d_F(x_1,x_2)=d_F(x_1,\partial\Omega)\leq i_F$ and $\gamma(t):[0,1]\to \overline{\Omega}$ the minimal geodesic connected $x_1$ with $x_2$. Using the gradient estimates (\ref{grad est2}), we have
\begin{eqnarray*}\frac{\pi}{2}- \arcsin (\frac{\alpha}{\alpha+1}) &=&\int_0^1 \frac{1}{\sqrt{(\alpha+1)^2-(\alpha+u)^2 }} du \leq \sqrt{\lambda_1}\int_0^1 \frac{1}{F(\nabla u)} du \\& \leq &\sqrt{\lambda_1}\int_0^1 \frac{1}{F(\nabla u(\gamma(t)))}\langle\nabla u(\gamma(t)), \dot{\gamma}(t)\rangle dt\\&\leq &\sqrt{\lambda_1}\int_0^1 F^0(\dot{\gamma}(t)) dt\leq \sqrt{\lambda_1} i_F.\end{eqnarray*}
Here we have used the Cauchy-Schwarz inequality (\ref{CS}).
Letting $\alpha\to 0$, we obtain
\begin{eqnarray*}
\lambda_1\geq \frac{\pi^2}{4i_F^2}.\end{eqnarray*} Thus we finish the proof of Theorem \ref{main thm}.


\

\end{document}